\documentclass[10pt]{amsart}
\usepackage{amssymb}
\usepackage{latexsym}
\usepackage{graphicx}
\usepackage{hyperref}
\usepackage{comment}

\newcommand{\Z}{{\mathbb Z}}
\newcommand{\R}{{\mathbb R}}
\newcommand{\C}{{\mathbb C}}

\newcommand{\N}{{\mathbb N}}

\newtheorem{lemma}{Lemma}[section]
\newtheorem{theorem}[lemma]{Theorem}

\newtheorem{prop}[lemma]{Proposition}
\newtheorem{coro}[lemma]{Corollary}
\newtheorem{defi}[lemma]{Definition}
\newtheorem*{conj}{Kotani-Last Conjecture}
\newtheorem*{conj2}{Kotani-Last Conjecture -- Dynamical Formulation}

\newcommand{\be}{\begin{equation}}
\newcommand{\ee}{\end{equation}}

\newcommand{\ol}{\overline}
\newcommand{\ti}{\tilde}

\DeclareMathOperator{\supp}{supp}

\numberwithin{equation}{section}

\begin{document}

\title[Schr\"odinger Operators defined by IET's]{Schr\"odinger Operators defined by Interval Exchange Transformations}

\author[J.\ Chaika]{Jon Chaika}

\address{Department of Mathematics, Rice University, Houston, TX~77005, USA}

\email{Jonathan.M.Chaika@rice.edu}

\author[D.\ Damanik]{David Damanik}

\address{Department of Mathematics, Rice University, Houston, TX~77005, USA}

\email{damanik@rice.edu}

\author[H.\ Kr\"uger]{Helge Kr\"uger}

\address{Department of Mathematics, Rice University, Houston, TX~77005, USA}

\email{helge.krueger@rice.edu}

\thanks{D.\ D.\ was supported in part by NSF grants DMS--0653720 and DMS--0800100.}
\thanks{H.\ K.\ was supported by NSF grant DMS--0800100.}
\thanks{Journal of Modern Dynamics \textbf{3:2}, 253-270 (2009)}

\date{\today}

\keywords{Ergodic Schr\"odinger operators, singular spectrum,
continuous spectrum, interval exchange transformations}
\subjclass[2000]{Primary 81Q10; Secondary  37A05, 47B36, 82B44}

\begin{abstract}
We discuss discrete one-dimensional Schr\"odinger operators whose
potentials are generated by an invertible ergodic transformation
of a compact metric space and a continuous real-valued sampling
function. We pay particular attention to the case where the
transformation is a minimal interval exchange transformation.
Results about the spectral type of these operators are
established. In particular, we provide the first examples of
transformations for which the associated Schr\"odinger operators
have purely singular spectrum for every non-constant continuous
sampling function.
\end{abstract}

\maketitle

\section{Introduction}\label{s.int}

Consider a probability space $(\Omega,\mu)$ and an invertible
ergodic transformation $T : \Omega \to \Omega$. Given a bounded
measurable sampling function $f : \Omega \to \R$, one can consider
discrete one-dimensional Schr\"odinger operators acting on $\psi
\in \ell^2(\Z)$ as
\begin{equation}\label{f.oper}
[H_\omega \psi](n) = \psi(n+1) + \psi(n-1) + V_\omega(n) \psi(n),
\end{equation}
where $\omega \in \Omega$, $n \in \Z$, and
\begin{equation}\label{f.pot}
V_\omega(n) = f(T^n \omega).
\end{equation}
Clearly, each $H_\omega$ is a bounded self-adjoint operator on
$\ell^2(\Z)$.

It is often convenient to further assume that $\Omega$ is a
compact metric space and $\mu$ is a Borel measure. In fact, one
can essentially force this setting by mapping $\Omega \ni \omega
\mapsto V_\omega \in \tilde \Omega$, where $\tilde \Omega$ is an
infinite product of compact intervals. Instead of $T$, $\mu$, and
$f$ one then considers the left shift, the push-forward of $\mu$,
and the function that evaluates at the origin. In the topological
setting, it is natural to consider continuous sampling functions.

A fundamental result of Pastur \cite{p} and Kunz-Souillard
\cite{ks} shows that there are $\Omega_0 \subseteq \Omega$ with
$\mu(\Omega_0) = 1$ and $\Sigma$, $\Sigma_\mathrm{pp}$,
$\Sigma_\mathrm{sc}$, $\Sigma_\mathrm{ac} \subset \R$ such that
for $\omega \in \Omega_0$, we have $\sigma(H_\omega) = \Sigma$,
$\sigma_\mathrm{pp}(H_\omega) = \Sigma_\mathrm{pp}$,
$\sigma_\mathrm{sc}(H_\omega) = \Sigma_\mathrm{sc}$, and
$\sigma_\mathrm{ac}(H_\omega) = \Sigma_\mathrm{ac}$. Here,
$\sigma(H)$, $\sigma_\mathrm{pp}(H)$, $\sigma_\mathrm{sc}(H)$, and
$\sigma_\mathrm{ac}(H)$ denote the spectrum, the pure-point
spectrum, the singular continuous spectrum, and the absolutely
continuous spectrum of $H$, respectively.

A standard direct spectral problem is therefore the following:
given the family $\{H_\omega\}_{\omega \in \Omega}$, identify the
sets $\Sigma$, $\Sigma_\mathrm{pp}$, $\Sigma_\mathrm{sc}$,
$\Sigma_\mathrm{ac}$, or at least determine which of them are
non-empty. There is a large literature on questions of this kind
and the area has been especially active recently; we refer the
reader to the survey articles \cite{d,j} and references therein.

Naturally, inverse spectral problems are of interest as well.
Generally speaking, given information about one or several of the
sets $\Sigma$, $\Sigma_\mathrm{pp}$, $\Sigma_\mathrm{sc}$,
$\Sigma_\mathrm{ac}$, one wants to derive information about the
operator family $\{H_\omega\}_{\omega \in \Omega}$. There has not
been as much activity on questions of this kind. However, a certain
inverse spectral problem looms large over the general area of
one-dimensional ergodic Schr\"odinger operators:

\begin{conj}
If $\Sigma_\mathrm{ac} \not= \emptyset$, then the potentials
$V_\omega$ are almost periodic.
\end{conj}

A function $V : \Z \to \R$ is called almost periodic if the set of
its translates $V_m(\cdot) = V(\cdot - m)$ is relatively compact
in $\ell^\infty(\Z)$. The Kotani-Last conjecture has been around
for at least two decades and it has been popularized by several
people, most recently by Damanik \cite[Problem~2]{d}, Jitomirskaya
\cite[Problem~1]{j}, and Simon \cite[Conjecture~8.9]{s}. A weaker
version of the conjecture is obtained by replacing the assumption
$\Sigma_\mathrm{ac} \not= \emptyset$ with $\Sigma_\mathrm{pp} \cup
\Sigma_\mathrm{sc} = \emptyset$. A proof of either version of the
conjecture would be an important result.

There is an equivalent formulation of the Kotani-Last conjecture
in purely dynamical terms. Suppose $\Omega$, $\mu$, $T$, $f$ are
as above. For $E \in \R$, consider the map
$$
A_E : \Omega \to \mathrm{SL}(2,\R) , \quad \omega \mapsto
\begin{pmatrix} E - f(\omega) & -1 \\ 1 & 0 \end{pmatrix}.
$$
This gives rise to the one-parameter family of
$\mathrm{SL}(2,\R)$-cocycles, $(T,A_E) : \Omega \times \R^2 \to
\Omega \times \R^2$, $(\omega,v) \mapsto (T \omega , A_E(\omega)
v)$. For $n \in \Z$, define the matrices $A_E^n$ by $(T,A_E)^n =
(T^n , A_E^n)$. Kingman's subadditive ergodic theorem ensures the
existence of $L(E) \ge 0$, called the Lyapunov exponent, such that
$
L(E) = \lim_{n \to \infty} \frac{1}{n} \int \log \|A_E^n(\omega)\|
\, d\mu(\omega).
$
Then, the following conjecture is equivalent to the one above;
see, for example, \cite[Theorem~4]{d}.
\begin{conj2}
Suppose the potentials $V_\omega$ are not almost periodic. Then,
$L(E) > 0$ for Lebesgue almost every $E \in \R$.
\end{conj2}

The existing evidence in favor of the Kotani-Last conjecture is
that its claim is true in the large number of explicit examples
that have been analyzed. On the one hand, there are several
classes of almost periodic potentials which lead to (purely)
absolutely continuous spectrum. This includes some limit-periodic
cases as well as some quasi-periodic cases. It is often helpful to
introduce a coupling constant and then work in the small coupling
regime. For example, if $\alpha$ is Diophantine, $T$ is the
rotation of the circle by $\alpha$, $g$ is real-analytic, and $f =
\lambda g$, then for $\lambda$ small enough, $\Sigma_\mathrm{pp}
\cup \Sigma_\mathrm{sc} = \emptyset$; see Bourgain-Jitomirskaya
\cite{bj} and references therein for related results. On the other
hand, all non-almost periodic operator families that have been
analyzed are such that $\Sigma_\mathrm{ac} = \emptyset$. For
example, if $\alpha$ is irrational and $T$ is the rotation of the
circle by $\alpha$, then the potentials $V_\omega(n) = f(\omega +
n\alpha)$ are almost periodic if and only if $f$ is continuous ---
and, indeed, Damanik and Killip have shown that
$\Sigma_\mathrm{ac} = \emptyset$ for discontinuous $f$ (satisfying
a weak additional assumption) \cite{dk}.

In particular, we note how the understanding of an inverse
spectral problem can be advanced by supplying supporting evidence
on the direct spectral problem side.

This point of view will be taken and developed further in this
paper. We wish to study operators whose potentials are not almost
periodic, but are quite close to being almost periodic in a sense
to be specified, and prove $\Sigma_\mathrm{ac} = \emptyset$. In
some sense, we will study a question that is dual to the
Damanik-Killip paper. Namely, what can be said when $f$ is nice
but $T$ is not? More precisely, can discontinuity of $T$ be
exploited in a similar way as discontinuity of $f$ was exploited
in \cite{dk}?

The primary example we have in mind is when the transformation $T$
is an interval exchange transformation. Interval exchange
transformations are important and extensively studied dynamical
systems. An interval exchange transformation is obtained by
partitioning the unit interval into finitely many half-open
subintervals and permuting them. Rotations of the circle
correspond to the exchange of two intervals. Thus, interval
exchange transformations are natural generalizations of circle
rotations. Note, however, that they are in general discontinuous.
Nevertheless, interval exchange transformations share certain
basic ergodic properties with rotations. First, every irrational
rotation is minimal, that is, all its orbits are dense. Interval
exchange transformations are minimal if they obey the Keane
condition, which requires that the orbits of the points of
discontinuity are infinite and mutually disjoint, and which holds
under a certain kind of irrationality assumption and in particular
in Lebesgue almost all cases. Secondly, every irrational rotation
is uniquely ergodic. This corresponds to the Masur-Veech result
that (for fixed irreducible permutation) Lebesgue almost every
interval exchange transformation is uniquely ergodic; see
\cite{m,v}. A major difference between irrational rotations and a
typical interval exchange transformation that is not a rotation is
the weak mixing property. While irrational rotations are never
weakly mixing, Avila and Forni have shown that (for every
permutation that does not correspond to a rotation) almost every
interval exchange transformation is weakly mixing \cite{af}.

This leads to an interesting general question: If the
transformation $T$ is weakly mixing, can one prove
$\Sigma_\mathrm{ac} = \emptyset$ for most/all non-constant $f$'s?
The typical interval exchange transformations form a prominent
class of examples. For these we are able to use weak mixing to
prove the absence of absolutely continuous spectrum for Lipschitz
functions:

\begin{theorem}\label{t.sample2}
Suppose $T$ is an interval exchange transformation that satisfies
the Keane condition and is weakly mixing. Then, for every
non-constant Lipschitz continuous $f$ and every $\omega$, we have
$\sigma_\mathrm{ac}(H_\omega) = \emptyset$.
\end{theorem}

For more general continuous functions, it is rather the
discontinuity of an interval exchange transformation that we
exploit and not the weak mixing property. We regard it as an
interesting open question whether the Lipschitz assumption in
Theorem~\ref{t.sample2} can be removed and the conclusion holds
for every non-constant continuous $f$.

Our results will include the following theorem, which we state
here because its formulation is simple and it identifies
situations where the absence of absolutely continuous spectrum
indeed holds for every non-constant continuous $f$.

\begin{theorem}\label{t.sample}
Suppose $r \ge 3$ is odd and $T$ is an interval exchange
transformation that satisfies the Keane condition and reverses the
order of the $r$ partition intervals. Then, for every non-constant
continuous $f$ and every $\omega$, we have
$\sigma_\mathrm{ac}(H_\omega) = \emptyset$.
\end{theorem}

\noindent\textit{Remarks.} (i) Reversal of order means that if we
enumerate the partition intervals from left to right by
$1,\ldots,r$, then $i < j$ implies that the image of the $i$-th
interval under $T$ lies to the right of the image of the $j$-th
interval. Note that our assumptions on $T$ leave a lot of freedom
for the choice of interval lengths. Indeed, the Keane condition
holds for Lebesgue almost all choices of interval lengths.
\\[1mm]
(ii) We would like to emphasize here that \emph{all} non-constant
continuous sampling functions are covered by this result. Of
course, if $f$ is constant, the potentials $V_\omega$ are constant
and the spectrum is purely absolutely continuous, so the result is
best possible as far as generality in $f$ is concerned.
\\[1mm]
(iii) Even more to the point, to the best of our knowledge,
Theorem~\ref{t.sample} is the first result of this kind, that is,
one that identifies an invertible transformation $T$ for which the
absolutely continuous spectrum is empty for all non-constant
continuous sampling functions.\footnote{There are non-invertible
examples, such as the doubling map, defining operators on
$\ell^2(\Z_+)$. In these cases, the absence of absolutely
continuous spectrum follows quickly from non-invertibility; see
\cite{dk2}.} While one should expect that it should be easier to
find such a transformation among those that are strongly mixing,
no such example is known. Our examples are all not strongly mixing
(this holds for interval exchange transformations in general)
\cite{nomixing}.\footnote{After a preliminary version of this
paper was posted, Svetlana Jitomirskaya explained to us how to
prove the absence of absolutely continuous spectrum for the
two-sided full shift and every non-constant continuous sampling
function. Her proof is also based on Kotani theory.}
\\[1mm]
(iv) In particular, this shows that interval exchange
transformations behave differently from rotations in that
regularity (or any other property) of $f$ and sufficiently small
coupling do not yield the existence of some absolutely continuous
spectrum.

The organization of the paper is as follows. We collect some
general results about the absolutely continuous spectrum of
ergodic Schr\"odinger operators in Section~\ref{s.gen}. In
Section~\ref{s.iet}, we recall the formal definition of an
interval exchange transformation and the relevance of the Keane
condition for minimality of such dynamical systems. The main
result established in Section~\ref{s.glob} is that for minimal
interval exchange transformations and continuous sampling
functions, the spectrum and the absolutely continuous spectrum of
$H_\omega$ are globally independent of $\omega$.
Section~\ref{s.ac} is the heart of the paper. Here we establish
several sufficient conditions for the absence of absolutely
continuous spectrum. Theorems~\ref{t.sample2} and \ref{t.sample}
will be particular consequences of these results. Additionally, we
establish a few results in topological dynamics. In
Section~\ref{s.pp} we use Gordon's lemma to prove almost sure
absence of eigenvalues for (what we call) Liouville interval
exchange transformations. This is an extension of a result of
Avron and Simon.

\section{Absolutely Continuous Spectrum: Kotani, Last-Simon, Remling}\label{s.gen}

In this section we discuss several important results concerning
the absolutely continuous spectrum. The big three are Kotani
theory \cite{k1,k2,k3,k4,k5}, the semicontinuity result of Last
and Simon \cite{ls}, and Remling's earthquake\footnote{For lack of
a better term, we quote Barry Simon here (OPSFA9, Luminy, France,
July 2007).} \cite{r}.

Kotani theory and Remling's work provide ways to use the presence of
absolutely continuous spectrum to ``predict the future.'' Kotani
describes this phenomenon by saying that \emph{all potentials
leading to absolutely continuous spectrum are deterministic}, while
Remling uses absolutely continuous spectrum to establish an
\emph{oracle theorem}. Last and Simon discuss the effect of taking
limits of translates of a potentials on the absolutely continuous
spectrum. Their result may also be derived within the context
Remling is working in; see \cite{r} for the derivation.

Let us return to our general setting where $(\Omega,\mu,T)$
ergodic is given and, for a bounded measurable function $f :
\Omega \to \R$, the potentials $V_\omega$ and operators $H_\omega$
are defined by \eqref{f.pot} and \eqref{f.oper}, respectively. We
first recall a central result from Kotani theory. It is convenient
to pass to the topological setting described briefly in the
introduction. Choose a compact interval $I$ that contains the
range of $f$ and set $\tilde \Omega = I^\Z$, equipped with the
product topology. The shift transformation $\tilde T : \tilde
\Omega \to \tilde \Omega$ is given by $[\tilde T \tilde \omega](n)
= \tilde \omega(n+1)$. It is clearly a homeomorphism. Consider the
measurable map $\mathcal{K} : \Omega \to \tilde \Omega$, $\omega
\mapsto V_\omega$. The push-forward of $\mu$ by the map
$\mathcal{K}$ will be denoted by $\tilde \mu$. Finally, the
function $\tilde f : \tilde \Omega \to \R$ is given by $\tilde
f(\tilde \omega) = \tilde \omega (0)$.

We obtain a system $(\tilde \Omega , \tilde \mu , \tilde T ,
\tilde f)$ that generates Schr\"odinger operators $\{\tilde
H_{\tilde \omega}\}_{\tilde \omega \in \tilde \Omega}$ with
potentials $\tilde V_{\tilde \omega}(n) = \tilde f(\tilde T^n
\tilde \omega)$ in such a way that the associated sets $\tilde
\Sigma$, $\tilde \Sigma_\mathrm{pp}$, $\tilde \Sigma_\mathrm{sc}$,
$\tilde \Sigma_\mathrm{ac}$ coincide with the original sets
$\Sigma$, $\Sigma_\mathrm{pp}$, $\Sigma_\mathrm{sc}$,
$\Sigma_\mathrm{ac}$, but they come from a homeomorphism $\tilde
T$ and a continuous function $\tilde f$.

The following theorem shows that the presence of absolutely
continuous spectrum implies (continuous) determinism. For proofs,
see \cite{k5} and \cite{d}, but its statement and history can be
traced back to the earlier Kotani papers \cite{k1,k2,k3,k4}.

\begin{theorem}[Kotani's Continuous Extension Theorem]\label{t.kotani1}
Suppose $\Sigma_\mathrm{ac} \not= \emptyset$. Then, the
restriction of any $\tilde \omega \in \supp \tilde \mu$ to either
$\Z_+$ or $\Z_-$ determines
$\tilde \omega$ uniquely among elements of $\supp \tilde \mu$ and
the extension map
$
\mathcal{E} : \supp \tilde \mu|_{\Z_-} \to \supp \tilde \mu ,
\quad \tilde \omega|_{\Z_-} \to \tilde \omega
$
is continuous.
\end{theorem}

Another useful result, due to Last and Simon \cite{ls}, is that the
absolutely continuous spectrum cannot shrink under pointwise
approximation by translates:

\begin{theorem}[Last-Simon Semicontinuity]\label{t.lastsimon}
Suppose that $\omega_1,\omega_2 \in \Omega$ and $n_k \to \infty$
are such that the potentials $V_{T^{n_k} \omega_1}$ converge
pointwise to $V_{\omega_2}$ as $k \to \infty$. Then,
$\sigma_\mathrm{ac}(H_{\omega_1}) \subseteq
\sigma_\mathrm{ac}(H_{\omega_2})$.
\end{theorem}

We will furthermore need the following proposition,
which follows from strong convergence.

\begin{prop} \label{sigess} Suppose that $\omega_1,\omega_2 \in \Omega$ and $n_k \to \infty$
are such that the potentials $V_{T^{n_k} \omega_1}$ converge
pointwise to $V_{\omega_2}$ as $k \to \infty$. Then,
$\sigma(H_{\omega_1}) \supseteq
\sigma(H_{\omega_2})$.
\end{prop}

Theorems~\ref{t.kotani1}, \ref{t.lastsimon} and
Proposition~\ref{sigess} will be sufficient for our purpose. We do
want to point out, however, that all the results that concern the
absolutely continuous spectrum, and much more, are proved by
completely different methods in Remling's paper \cite{r} --- hence
the term \emph{earthquake}.

\section{Interval Exchange Transformations}\label{s.iet}

In this section we give a precise definition of an interval
exchange transformation and gather a few known facts about these
maps. As a general reference, we recommend Viana's survey
\cite{via}.

\begin{defi}
Suppose we are given an integer $r \ge 2$, a permutation $\pi:
\{1, \dots, r\} \to \{1, \dots r\}$, and an element
$\Lambda \in \Delta^r = \{(\lambda_1, \dots, \lambda_r), \lambda_i
> 0, \sum_{j=1}^r \lambda_j = 1\}$.
Let $\Omega = [0,1)$ be the half-open unit interval with the
topology inherited from $\R$. Then the interval exchange
transformation $T : \Omega \to \Omega$ associated with $(\pi,
\Lambda)$ is given by $T: I_j \ni \omega \mapsto \omega -
\sum_{i=1}^{j-1} \lambda_i + \sum_{i,\,\pi(i) < \pi(j)}
\lambda_i$, where $I_j = \sum_{i=1}^{j-1} \lambda_i + [0,
\lambda_j)$.
\end{defi}

\noindent\textit{Remarks.}
(i) If $r = 2$ and $\pi(1) = 2$, $\pi(2) = 1$, then $T$ is
conjugate to $\tilde T : \R / \Z \to \R / \Z$, $\tilde \omega
\mapsto \tilde \omega + \lambda_2$, that is, a rotation of the
circle. More generally, the same is true whenever $\pi(j) - 1 = j +
k \mod r$ for some $k$. In this sense, interval exchange
transformations are natural generalizations of rotations of the
circle. We call $\pi$ with this property of rotation class.
\\[1mm]
(ii) The permutation $\pi$ is called reducible if there is $k <
r$ such that $\pi(\{1,\ldots,k\}) = \{1,\ldots,k\}$ and irreducible
otherwise. If $\pi$ is reducible, $T$ splits into two interval
exchange transformations. For this reason, it is natural to only
consider irreducible permutations.

\begin{defi}
The interval exchange transformation associated with
$(\pi,\Lambda)$ satisfies the Keane condition if the orbits of the
left endpoints of the intervals $I_1, \ldots,I_r$ are infinite and
mutually disjoint.
\end{defi}

The following result is \cite[Proposition~3.2,4.1]{via}):

\begin{prop}[Keane]\label{p.keane1}
If $\pi$ is irreducible and $\Lambda$ is rationally independent,
then the interval exchange transformation associated with
$(\pi,\Lambda)$ satisfies the Keane condition.

If an interval exchange transformation satisfies the Keane
condition, then it is minimal. More precisely, for every $\omega
\in \Omega$, the set $\{ T^n \omega : n \ge 1 \}$ is dense in
$\Omega$.
\end{prop}

On the other hand, the Keane condition is not necessary for
minimality; see \cite[Remark~4.5]{via} for an example. For a more
detailed discussion, see \cite[Section~2]{rank2}.

\begin{lemma} \label{lem:helg}
 Assume that $T$ is minimal. Given any $N \geq 1$, $\omega, \omega' \in \Omega$, $\epsilon > 0$,
 there exists $l \in \Z$ such that
 \be\label{eq:hap}
  |T^m(\omega)-T^{m+l}(\omega')|<\epsilon
 \ee
 for any $-N \leq m \leq N$.
\end{lemma}

\begin{proof}
Since the set of discontinuities of $T^j$, $-N \leq j \leq N$ is
discrete in $[0,1)$ and we have chosen $T$ to be right continuous,
we can find an interval $[a,b) \subseteq [0,1)$ containing
$\omega$ such that $T^j$ are isometries on $[a,b)$ for $-N \leq j
\leq N$. Hence, for any $\ti{\omega} \in I := [a,b) \cap (\omega -
\epsilon, \omega + \epsilon)$, \eqref{eq:hap} holds. Since $T$ is
minimal and $I$ has non-empty interior, we can find $l$ such that
$T^l \omega' \in I$, finishing the proof.
\end{proof}

\section{Some Facts About Schr\"odinger Operators Generated by Interval Exchange
Transformations}\label{s.glob}

In this section, we prove some general results for potentials and
operators generated by minimal interval exchange transformations.
First, we show that all potentials belong to the support of the
induced measure and then we use this observation to show that the
absolutely continuous spectrum is globally constant, not merely
almost surely with respect to a given ergodic measure. Note that
there are minimal interval exchange transformations that admit
several ergodic measures.

\begin{lemma}\label{l.suppincl}
Suppose that $T$ is a minimal interval exchange transformation and
$\mu$ is $T$-ergodic. Then, for every continuous sampling function
$f$, we have $\{ V_\omega : \omega \in \Omega \} \subseteq \supp
\tilde \mu$, where $\ti{\mu}$ is as defined in
Section~\ref{s.gen}.
\end{lemma}

\begin{proof}
Note first that since $\supp \mu$ is closed, non-empty, and
$T$-invariant, it must equal all of $\Omega$ because $T$ is
minimal. Now given any $\omega \in \Omega$, we wish to show that
$V_\omega \in \supp \tilde \mu$. Consider the sets $[\omega ,
\omega + \varepsilon)$ for $\varepsilon > 0$ small. Since $\supp
\mu = \Omega$, $\mu ( [\omega,\omega + \varepsilon) )
> 0$ and hence there exists $\omega_\varepsilon \in [\omega,\omega
+ \varepsilon)$ such that $V_{\omega_\varepsilon} \in \supp \tilde
\mu$. It follows that there is a sequence $\omega_n$ which
converges to $\omega$ from the right and for which $V_{\omega_n}
\in \supp \tilde \mu$. Since $T$ is right-continuous and $f$ is
continuous, it follows that $V_{\omega_n} \to V_\omega$ pointwise
as $n \to \infty$. Since $\supp \tilde \mu$ is closed, it follows
that $V_\omega \in \supp \tilde \mu$.
\end{proof}


\begin{theorem}\label{t.constantac}
 Suppose that $T$ is a minimal interval exchange transformation. Then, for
 every continuous sampling function $f$, we have that
 $\sigma(H_\omega)$ and $\sigma_\mathrm{ac}(H_\omega)$ are independent of $\omega$.
\end{theorem}

\begin{proof}
We first show that for $\omega, \omega' \in \Omega$, we have
$\sigma(H_{\omega}) = \sigma(H_{\omega'})$. By
Lemma~\ref{lem:helg} and $f$ being continuous, we can find a
sequence $n_k$ such that $\sup_{|l| \leq k} |f(T^{n_k + l} \omega)
- f(T^{l} \omega')| \leq \frac{1}{k}$. Hence $V_{T^{n_k} \omega}$
converges pointwise to $V_{\omega'}$. Now Proposition~\ref{sigess}
implies that $\sigma(H_{\omega}) = \sigma(H_{T^{n_k} \omega})
\supseteq \sigma(H_{\omega'})$. Interchanging the roles of
$\omega$ and $\omega'$ now finishes the proof. To prove the claim
about the absolutely continuous spectrum, repeat the above
argument replacing Proposition~\ref{sigess} by
Theorem~\ref{t.lastsimon} to obtain
$\sigma_\mathrm{ac}(H_{\omega}) = \sigma_\mathrm{ac}(H_{T^{n_k}
\omega}) \subseteq \sigma_\mathrm{ac}(H_{\omega'})$. This implies
$\omega$ independence of $\sigma_\mathrm{ac}(H_\omega)$.
\end{proof}

\section{Absence of Absolutely Continuous Spectrum for Schr\"odinger
Operators Generated by Interval Exchange
Transformations}\label{s.ac}

In this section we identify situations in which the absolutely
continuous spectrum can be shown to be empty. Particular
consequences of the results presented below are the sample
theorems stated in the introduction, Theorems~\ref{t.sample2} and
\ref{t.sample}. We will make explicit later in this section how
these theorems follow from the results obtained here.

Let us first establish the tool we will use to exclude absolutely
continuous spectrum.

\begin{lemma}\label{l.mainlemma}
Suppose $T$ is a minimal interval exchange transformation and $f$
is a continuous sampling function. Assume further that there are
points $\omega_k$, $\hat \omega_k \in \Omega$ such that for some
$N \geq 1$, we have
\begin{equation}\label{eq:condkot1}
\limsup_{k\to\infty} |f(T^N \omega_k) - f(T^N \hat \omega_k)| > 0,
\end{equation}
and for every $n \leq 0$, we have
\begin{equation}\label{eq:condkot2}
\lim_{k\to\infty} |f(T^n \omega_k) - f(T^n \hat \omega_k)| = 0.
\end{equation}
Then, $\sigma_\mathrm{ac}(H_\omega) = \emptyset$ for every $\omega
\in \Omega$.
\end{lemma}

\begin{proof}
By passing to a subsequence of $\omega_k$, we can assume that the
limit in \eqref{eq:condkot1} exists. Fix any $T$-ergodic measure
$\mu$ and assume that the corresponding almost sure absolutely
continuous spectrum, $\Sigma_\mathrm{ac}$, is non-empty. By
Lemma~\ref{l.suppincl}, the potentials corresponding to the
$\omega_k$'s and $\hat \omega_k$'s belong to $\supp \tilde \mu$.
Then, by Theorem~\ref{t.kotani1}, the map from the restriction of
elements of $\supp \tilde \mu$ to their extensions is continuous.
Moreover, since the domain of this map is a compact metric space,
the map is uniformly continuous. Combining this with
\eqref{eq:condkot1} and \eqref{eq:condkot2}, we arrive at a
contradiction. Thus, the $\mu$-almost sure absolutely continuous
spectrum is empty. This implies the assertion because
Theorem~\ref{t.constantac} shows that the absolutely continuous
spectrum is independent of the parameter.
\end{proof}

Our first result on the absence of absolutely continuous spectrum
works for all permutations, but we have to impose a weak
assumption on~$f$.

\begin{theorem}\label{t.generalthm}
Suppose $T$ is an interval exchange transformation that satisfies
the Keane condition. If $f$ is a continuous sampling function such
that there are $N \ge 1$ and $\omega_d \in (0,1)$ with
\begin{equation}\label{f.maincond}
 \lim_{\omega \uparrow \omega_d} f(T^N \omega) \neq \lim_{\omega \downarrow \omega_d} f(T^N
 \omega),
\end{equation}
then $\sigma_\mathrm{ac}(H_\omega) = \emptyset$ for every $\omega
\in \Omega$. In other words, if $f$ is continuous and $f \circ
T^n$ is discontinuous for some $n \ge 1$, then
$\sigma_\mathrm{ac}(H_\omega) = \emptyset$ for every $\omega \in
\Omega$.
\end{theorem}

\begin{proof}
Since $f$ is continuous, it follows from \eqref{f.maincond} that
$\omega_d$ is a discontinuity point of $T^N$. Since $T$ satisfies
the Keane condition, this implies that $T^n$, $n \le 0$, are all
continuous at $\omega_d$. Thus, choosing a sequence $\omega_k
\uparrow \omega_d$ and a sequence $\hat \omega_k \downarrow
\omega_d$, we obtain the conditions \eqref{eq:condkot1} and
\eqref{eq:condkot2}. Thus, the theorem follows from
Proposition~\ref{p.keane1} and Lemma~\ref{l.mainlemma}.
\end{proof}

\noindent\textit{Remarks.} (i) If $\pi$ is irreducible and not of rotation class,
then the interval exchange transformation $T$ has discontinuity points. Fix any such point
$\omega_d$ and consider the continuous functions $f$ that obey
\begin{equation}\label{f.maincond2}
\lim_{\omega \uparrow \omega_d} f(T \omega) \neq \lim_{\omega
\downarrow \omega_d} f(T \omega).
\end{equation}
This is clearly an open and dense set, even in more restrictive
categories such as finitely differentiable or infinitely
differentiable functions. Thus, the condition on $f$ in
Theorem~\ref{t.generalthm} is rather weak.
\\[1mm]
(ii) If the permutation $\pi$ is such that $T$ is topologically conjugate to a
rotation of the circle, then there is only one discontinuity point
and the condition \eqref{f.maincond2} forces $f$ to be
discontinuous if regarded as a function on the circle! Thus, in
this special case, Theorem~\ref{t.generalthm} only recovers a
result of Damanik and Killip from \cite{dk}.

Even though the condition on $f$ in Theorem~\ref{t.generalthm} is
weak, it is of course of interest to identify cases where no
condition (other than non-constancy) has to be imposed at all. The
following definition will prove to be useful in this regard.

\begin{defi}
We define a finite directed graph $G = (V,E)$ associated with an
interval exchange transformation $T$ coming from data
$(\pi,\Lambda)$ with $\pi$ irreducible. Denote by $\omega_i$ the
right endpoint of $I_j$, $1 \le j \le r-1$. Set $V = \{\omega_1 ,
\ldots , \omega_{r-1}\} \cup \{0,1\}$. The edge set $E$ is defined
as follows. Write
$
T_- \hat \omega = \lim_{\omega \uparrow \hat \omega} T \omega
\quad \text{ and } \quad T_+ \hat \omega = \lim_{\omega \downarrow
\hat \omega} T \omega.
$
Then there is an edge from vertex $v_1$ to vertex $v_2$ if $T_-
v_1 = T_+ v_2$. In addition, there are two special edges, $\tilde
e_1$ and $\tilde e_2$. The edge $\tilde e_1$ starts at $0$ and
ends at $\omega_j$ for which $T_+ \omega_j = 0$, and the edge
$\tilde e_2$ ends at $1$ and starts at $\omega_k$ where $T_-
\omega_k = 1$.
\end{defi}

\noindent\textit{Remarks.} (i) Irreducibility of $\pi$ ensures
that the $\omega_j$ and $\omega_k$ that enter the definition of
the special edges actually exist.
\\[1mm]
(ii) Note that the graph $G$ depends only on $\pi$, that is, it is
independent of $\Lambda$. Indeed,
$$
 T_+(\omega_j) = \begin{cases} 0 & \pi(j + 1) = 1 \\ T_-(1) & \pi(j + 1) - 1 = \pi(r) \\  T_-(\omega_k) & \pi(j + 1) -1 = \pi(k) \end{cases}
 \quad \text{ and } \quad
 T_-(\omega_j) = \begin{cases} 1 & \pi(j) = r \\ T_+(0) & \pi(j) + 1 = \pi(1)  \\ T_+(\omega_k) & \pi(j)+1 = \pi(k+1) \end{cases}
$$
for $j = 1, \dots, r-1$.
\\[1mm]
(iii) $T$ is defined to be right continuous, so we could simply
replace $T_+ \hat{\omega}$ by $T\hat{\omega}$. In order to emphasize
that the discontinuity of $T$ is important, we did not do so in
the above definition.

\noindent\textit{Examples.} In these examples, we will denote
permutations $\pi$ as follows:
$$
\pi = \begin{pmatrix} 1 & 2 & 3 & \cdots & r \\ \pi(1) & \pi(2) &
\pi(3) & \cdots & \pi(r) \end{pmatrix}.
$$
(i) Here are the permutations relevant to Theorem~\ref{t.sample}:
$$
\pi = \begin{pmatrix} 1 & 2 & 3 & \cdots & 2k & 2k+1 \\ 2k+1 & 2k
& 2k-1 & \cdots & 2 & 1 \end{pmatrix}.
$$
The corresponding graph

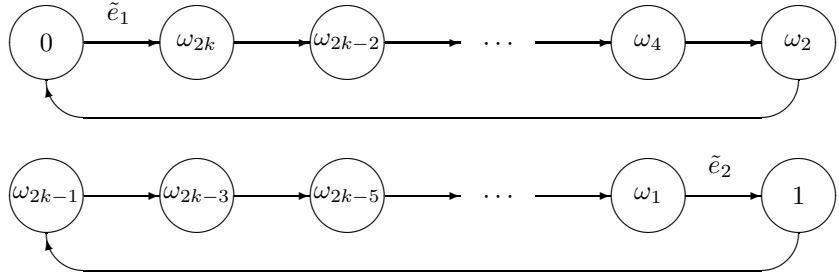
\setlength{\unitlength}{1mm}
\begin{center}
\begin{picture}(100,20)

\put(0,13){\circle{10}\makebox(0,0){$0$}}

\put(5,13){\vector(1,0){10}}

\put(20,13){\circle{10}\makebox(0,0){$\omega_{2k}$}}

\put(25,13){\vector(1,0){10}}

\put(40,13){\circle{10}\makebox(0,0){$\omega_{2k-2}$}}

\put(45,13){\vector(1,0){10}}

\put(65,13){\vector(1,0){10}}

\put(80,13){\circle{10}\makebox(0,0){$\omega_{4}$}}

\put(85,13){\vector(1,0){10}}

\put(100,13){\circle{10}\makebox(0,0){$\omega_{2}$}}

\put(58,12){$\cdots$}

\put(8,16){$\tilde e_1$}

\put(50,8){\oval(100,10)[b]}

\put(0.04,7.17){\vector(-1,4){0.01}}

\end{picture}

\begin{picture}(100,20)

\put(0,13){\circle{10}\makebox(0,0){$\omega_{2k-1}$}}

\put(5,13){\vector(1,0){10}}

\put(20,13){\circle{10}\makebox(0,0){$\omega_{2k-3}$}}

\put(25,13){\vector(1,0){10}}

\put(40,13){\circle{10}\makebox(0,0){$\omega_{2k-5}$}}

\put(45,13){\vector(1,0){10}}

\put(65,13){\vector(1,0){10}}

\put(80,13){\circle{10}\makebox(0,0){$\omega_{1}$}}

\put(85,13){\vector(1,0){10}}

\put(100,13){\circle{10}\makebox(0,0){$1$}}

\put(58,12){$\cdots$}

\put(88,16){$\tilde e_2$}

\put(50,8){\oval(100,10)[b]}

\put(0.04,7.17){\vector(-1,4){0.01}}

\end{picture}
\end{center}
consists of two cycles, each of which contains exactly one special
edge.
\\[1mm]
(ii) For the permutations
$$
\pi = \begin{pmatrix} 1 & 2 & \cdots & k & k+1 & \cdots & r-1 & r \\
r-k+1 & r-k+2 & \cdots & r & 1 & \cdots & r-k-1 & r-k
\end{pmatrix},
$$
which correspond to rotations of the circle, the graph looks as
follows:
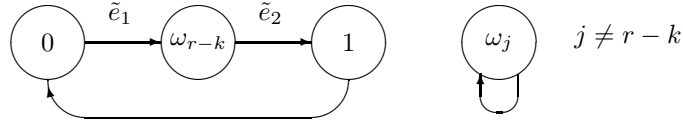
\begin{center}
\begin{picture}(105,20)

\put(10,13){\circle{10}\makebox(0,0){$0$}}

\put(15,13){\vector(1,0){10}}

\put(30,13){\circle{10}\makebox(0,0){$\omega_{r-k}$}}

\put(35,13){\vector(1,0){10}}

\put(50,13){\circle{10}\makebox(0,0){$1$}}

\put(18,16){$\tilde e_1$}

\put(38,16){$\tilde e_2$}

\put(30,8){\oval(40,10)[b]}

\put(10.04,7.17){\vector(-1,4){0.01}}

\put(70,13){\circle{10}\makebox(0,0){$\omega_{j}$}}

\put(80,13){$j \not= r-k$}

\put(70,8.7){\oval(5,10)[b]}

\put(67.506,7.82){\vector(0,1){1}}

\end{picture}
\end{center}
It consists of several cycles, each of which contains either no
special edge or two special edges.

\begin{theorem}\label{t.specialthm}
Suppose the permutation $\pi$ is irreducible and the associated
graph $G$ has a cycle that contains exactly one special edge.
Assume furthermore that $\Lambda$ is such that the interval
exchange transformation $T$ associated with $(\pi,\Lambda)$
satisfies the Keane condition. Then, for every continuous
non-constant sampling function $f$, we have
$\sigma_\mathrm{ac}(H_\omega) = \emptyset$ for every $\omega \in
\Omega$.
\end{theorem}

\begin{proof}
Our goal is to show that under the assumptions of the theorem,
every non-constant continuous sampling function satisfies
\eqref{f.maincond} for $N \ge 1$ and $\omega_d \in (0,1)$ suitable
and hence the assertion follows from Theorem~\ref{t.generalthm}.
So let us assume that $f$ is continuous and \eqref{f.maincond}
fails for all $N \ge 1$ and $\omega_d \in (0,1)$. We will show
that $f$ is constant. If there is a non-special edge from $v_1$ to
$v_2$, where $v_1 \not= v_2$, the Keane condition implies that for
every $n \ge 1$, $T^n$ will be continuous at the point $T_- v_1 =
T_+ v_2$ and hence $T^n_- v_1 = T^n_+ v_2$ for every $n \ge 1$,
where we generalize the previous definition as follows, $ T^n_-
\hat \omega = \lim_{\omega \uparrow \hat \omega} T^n \omega \quad
\text{ and } \quad T^n_+ \hat \omega = \lim_{\omega \downarrow
\hat \omega} T^n \omega. $ Next let us consider the special edges.
For the edge $\tilde e_1$ from $0$ to $\omega_j$ (for $j$
suitable), we have by construction and the Keane condition again
that $T^{n+1}_+ \omega_j = T^n_+ 0$ for $n \ge 1$. Likewise, for
the special edge $\tilde e_2$ from $\omega_k$ to $1$ (for $k$
suitable), we have $T^{n+1}_- \omega_k = T^n_- 1$ for $n \ge 1$.
By our assumption on $f$, the function $\omega \mapsto f(T^n
\omega)$ is continuous for every $n \ge 0$. In particular, we have
$f(T^n_- \omega_m) = f(T^n_+ \omega_m)$ for every $n \ge 0$ and $1
\le m \le r-1$. Consequently, if there is a path containing no
special edges connecting vertex $v_1$ to vertex $v_2$, then
$f(T_{\pm}^n(v_1))=f(T_{\pm}^n(v_2))$. On the other hand, if there
is a path containing just $\tilde e_1$ (but not $\tilde e_2$) from
$v_1$ to $v_2$, then $T_{\pm} v_1 = T_{\pm}^2 v_2$. Likewise, if
there is a path containing just $\tilde e_2$ from $v_1$ to $v_2$,
then $T_{\pm}^2(v_1)=T_{\pm}(v_2)$. Putting all these observations
together, we may infer that if there is a cycle containing exactly
one special edge and a vertex $\omega_j$, then $f(T^n \omega_j) =
f(T^{n+1} \omega_j)$ for every $n \ge 1$. The Keane condition
implies that the set $\{ T^n \omega_j : n \ge 1 \}$ is dense; see
Proposition~\ref{p.keane1}. Thus, the continuous function $f$ is
constant on a dense set and therefore constant.
\end{proof}

\begin{proof}[Proof of Theorem~\ref{t.sample}.]
Example~(i) above shows that Theorem~\ref{t.sample} is a special
case of Theorem~\ref{t.specialthm}.
\end{proof}

\noindent\textit{Remarks.} (i) Of course, it is necessary to
impose some condition on the permutation $\pi$ aside from
irreducibility in Theorem~\ref{t.specialthm}. It is known, and
discussed in the introduction, that the assertion of
Theorem~\ref{t.specialthm} fails for permutations corresponding to
rotations. Example~(ii) above shows the graph associated with such
permutations and demonstrates in which way the assumption of
Theorem~\ref{t.specialthm} fails in these cases.
\\[1mm]
(ii) There is another, more direct, way of seeing that
Theorem~\ref{t.generalthm} is not always applicable. Suppose that
the interval exchange transformation $T$ is not topologically
weakly mixing. This means that there is a non-constant continuous
function $g : \Omega \to \C$ and a real $\alpha$ such that $g(T
\omega) = e^{i \alpha} g(\omega)$ for every $\omega \in \Omega$.
At least one of $\Re g$ and $\Im g$ is non-constant. Consider the
case where $\Re g$ is non-constant, the other case is analogous.
Thus, $f = \Re g$ is a non-constant continuous function from
$\Omega$ to $\R$ such that
$
f(T^N \omega) = \Re g(T^N \omega) = \Re \left( e^{i \alpha}
g(T^{N-1} \omega) \right) = \cdots = \Re \left( e^{i N \alpha}
g(\omega) \right).
$
Since the right-hand side is continuous in $\omega$, so is the
left-hand side, and hence \eqref{f.maincond} fails for all $N \ge
1$ and $\omega_d \in (0,1)$. Nogueira and Rudolph proved that for
every permutation $\pi$ that does not generate a rotation and
Lebesgue almost every $\Lambda$, the transformation $T$ generated
by $(\pi,\Lambda)$ is topologically weakly mixing \cite{nr}. This
result was strengthened, as was pointed out earlier, by Avila and
Forni \cite{af}. However, there are cases where topological weak
mixing fails (see, e.g., \cite{h}) and hence the argument above
applies in these cases.

This example is even more striking, since if we write
$g(\omega) = \lambda e^{i \theta}$, we see that
$
 f(T^N \omega) = \lambda \Re (e^{i N \alpha + i \theta}) = \lambda \cos(\theta + N\alpha).
$
Hence the operator defined by \eqref{f.oper} is the almost Mathieu
operator and it has purely absolutely continuous spectrum for
$|\lambda| < 2$ for $\alpha$ irrational, and for all $\lambda$
if $\alpha$ is rational; see \cite{a} and references therein for earlier
partial results.

Let us briefly discuss the results above from a dynamical
perspective. To do so, we recall the definition of a \textit{Type
W} permutation; compare \cite[Definition 3.2]{cn}.

\begin{defi}
Suppose $\pi$ is an irreducible permutation of $r$ symbols. Define
inductively a sequence $\{a_k\}_{k = 0,\ldots, s}$ as follows. Set
$a_0 = 1$. If $a_k \in \{ \pi^{-1}(1) , r + 1\}$, then set $s = k$
and stop. Otherwise let $a_{k+1} = \pi^{-1}(\pi(a_k) - 1) + 1$.
The permutation $\pi$ is of \textit{Type W} if $a_s =
\pi^{-1}(1)$.
\end{defi}

It is shown in \cite[Lemma~3.1]{cn} that the process indeed
terminates and hence $a_s$ is well defined. Observing
\cite[Theorem~3.5]{cn}, \cite[Theorem~1.11]{metric} then reads as
follows:

\begin{theorem}[Veech 1984]
If $\pi$ is of \textit{Type W}, then for almost every $\Lambda$,
the interval exchange transformation associated with
$(\pi,\Lambda)$ is weakly mixing.
\end{theorem}

We have the following related result, which has a weaker
conclusion but a weaker, and in particular explicit, assumption:

\begin{coro}
Every interval exchange transformation satisfying the Keane
condition and associated to a Type~W permutation is topologically
weakly mixing.
\end{coro}

\begin{proof}
We claim that $\pi$ is Type W if and only if the associated graph
$G$ has a cycle that contains exactly one special edge. Assuming
this claim for a moment, we can complete the proof as follows.
Suppose $\Lambda$ is such that the Keane condition holds for the
interval exchange transformation $T$ associated with
$(\pi,\Lambda)$. We may infer from the proof of
Theorem~\ref{t.specialthm} that for every non-constant continuous
function $f$, there is an $n \ge 1$ such that $f \circ T^n$ is
discontinuous. This shows that $T$ does not have any continuous
eigenfunctions and hence it is topologically weakly mixing. Thus,
it suffices to prove the claim. Note that the procedure generating
the sequence $\{a_k\}_{k = 0,\ldots, s}$ is the same as
determining the arrows in the cycle containing the vertex $0$ in
the graph $G$ associated with $\pi$. Thus, if $\pi$ is not Type W,
then before we close the cycle containing the vertex $0$, it is
identified to the vertex $1$. This necessitates that this cycle
contains both special edges. On the other hand, if $\pi$ is Type
W, then the cycle containing the vertex $0$ closes up with just
one special edge.
\end{proof}

Recall now that a continuous map $S: Y \to Y$ is said to be a
topological factor of $T: [0,1) \to [0,1)$ if there is a
surjective continuous map $\pi: [0,1) \to Y$ such that $\pi \circ
T = S \circ \pi$. $T$ is called topologically prime if it has no
topological factors. If $T$ has a factor, we can find a continuous
function $f: [0,1) \to \R$ such that also $f \circ T^n$ is
continuous, by setting $f = g \circ \pi$ for some continuous $g: Y
\to \R$. Hence, Theorem~\ref{t.specialthm} also has the following
consequence:

\begin{coro}\label{cor:veech}
 Every interval exchange transformation satisfying the Keane
 condition and associated to a Type~W permutation is topologically
 prime.
\end{coro}

Let us now return to our discussion of sufficient conditions for
the absence of absolutely continuous spectrum, presenting results
 that hold under additional assumptions on the sampling function $f$.

\begin{theorem}
Assume that the graph $G$ contains a path with $\ell$ vertices
corresponding to distinct discontinuities. Then if the continuous function
$f$ is such that for every $x \in \R$, the set $f^{-1}(\{x\})$ has
at most $\ell-1$ elements, we have $\Sigma_\mathrm{ac} =
\emptyset$ for any $T$ satisfying the Keane condition.
\end{theorem}

\begin{proof}
Denote by $\widetilde{\Omega}$ the set of $\ell$ distinct vertices,
from the assumption of the theorem. If \eqref{f.maincond} fails, we see as in the
proof of the last theorem that $f(\omega) \in \{ f(\omega'), f(T^{-1}\omega'), f(T \omega' ) \}$ for
$\omega, \omega' \in \widetilde{\Omega}$. By the Keane condition this contradicts our
assumption of the preimage of a single point under $f$ containing
a maximum of $\ell - 1$ points.
\end{proof}

\begin{theorem}\label{t.lip}
Suppose that $f$ is non-constant and Lipschitz continuous and $T$
is an interval exchange transformation for which $f \circ T^n$ is
continuous for every $n$. Then, $T$ is not measure theoretically weakly mixing.
\end{theorem}

\begin{proof}
By assumption $f$ satisfies $|f(x) - f(y)| \leq k |x - y|$ for all
$x, y$. Assume that $f \circ T^n$ is continuous for every $n$.
Consider $x,y \in \Omega$ with $x<y$. Denote by $\delta_1 < \cdots
< \delta_t$ the discontinuities of $T^n$ between $x$ and $y$.
Then,
\begin{align*}
|f(T^n(x)) - f(T^n(y))| & \leq |f(T^n(x)) - f(T^n_-(\delta_1))| +
|f(T^n_-(\delta_1) - f(T^n_+(\delta_1))| \\
& \qquad + |f(T^n_+(\delta_1)) - f(T^n_-(\delta_2))|+ \cdots + |f(T_+(\delta_t)) - f(T^n(y))| \\
& = |f(T^n(x)) - f(T^n_-(\delta_1))| + |f(T^n_+(\delta_1)) -
f(T^n_-(\delta_2))| \\
& \qquad + \cdots + |f(T_+(\delta_t)) - f(T^n(y))| \\
& \le k(|T^n(x) - T^n_-(\delta_1)| + |T^n_+(\delta_1) -
T^n_-(\delta_2)| \\
& \qquad + \cdots + |T_+^n(\delta_t) - T^n(y)|) \\
& = k ( |x - \delta_1| + |\delta_1 - \delta_2| + \cdots + |\delta_t - y|)  = k (y-x),
\end{align*}
where we used in the first equality that $f(T_-^n(\delta_i)) =
f(T_+^n(\delta_i))$ since $f \circ T^n$ is continuous, in the
second inequality that $f$ is $k$-Lipschitz, and in the second
equality that $T^n$ is an isometry on $(\delta_i,\delta_{i+1})$.
So we have seen that, we also have that $|f(T^n x) - f(T^n y)|
\leq k |x - y|$. We now proceed to show that $T$ cannot be weak
mixing. Since $f$ is continuous and non-constant, there exist
$c>1$, $E$, $d$, such that $ \frac{1}{10} < \mathrm{Leb}
(\{x:|f(x)-E|<d\}) < \mathrm{Leb} (\{x:|f(x)-E|<cd\}) <
\frac{1}{2}. $ For any $\omega$ with $T^n(\omega) \in
\{x:|f(x)-z|<d\}$ and any $\ti{\omega} \in
B(\omega,\frac{(c-1)d}{k})$, we compute
$$
|f(T^n \ti{\omega}) - E| \leq |f(T^n \omega) - E| + |f(T^n \omega)
- f(T^n \ti{\omega})| < d + k \cdot \frac{(c-1) d}{k}  = cd,
$$
where we used that $f \circ T^n$ is $k$-Lipschitz continuous by
the last lemma. So, we have that $T^n \ti{\omega} \in \{x:|f(x)-E|
< cd\}$. This implies that for the set $ \left\{ (\omega,
\ti{\omega}) : \ti{\omega} \in B\left( \omega,\frac{(c-1)d}{k}
\right) \right\} \subseteq \Omega\times\Omega, $ which has
positive measure, we have that its image under $T^n \times T^n$
never intersects the set $\{\omega :|f(\omega)-E|<d\} \times
\{\ti{\omega}:|f(\ti{\omega})-E| \geq cd\}$, which also has
positive measure. Hence, $T \times T$ cannot be ergodic, and so
$T$ is not weakly mixing.
\end{proof}

\noindent \textit{Remark.} We only use that T is a piecewise
isometry and $f \circ T^n$ is continuous for all $n$. In
particular, the argument extends to rectangle exchange
transformations and interval translation maps.

\begin{proof}[Proof of Theorem~\ref{t.sample2}.]
The result follows from Theorems~\ref{t.generalthm} and
\ref{t.lip}.
\end{proof}

Let us say that a function $f$ has a non-degenerate maximum at
$\omega^{\mathrm{max}}$ if for any $\varepsilon > 0$, there is a
$\delta > 0$ such that \be\label{eq:condmax}
 f(\omega) \geq f(\omega^{\mathrm{max}}) - \delta \Rightarrow |\omega - \omega^{\mathrm{max}}| \leq \varepsilon.
\ee

\begin{theorem} \label{max}
Let $\pi$ be a non-trivial permutation.
Assume that $T = (\pi, \Lambda)$ satisfies the Keane condition and that $f$ has a
non-degenerate maximum at $\omega^{\mathrm{max}}$ and a continuous
bounded derivative. Then, \eqref{f.maincond} holds. In particular,
we have $\sigma_\mathrm{ac}(H_\omega) = \emptyset$ for every
$\omega \in \Omega$.
\end{theorem}

\begin{proof}
First, it follows from Keane's condition that we can assume (by
possibly replacing $T$ with $T^{-1}$) that $\omega^{\mathrm{max}}$
is not a discontinuity point of $T^n$, $n\geq 1$. Assume that
\eqref{f.maincond} fails, and fix a discontinuity point
$\omega^{\mathrm{disc}}$ of $T$. Our goal is to show that $f$ is
constant, which of course contradicts the assumption that the
maximum of $f$ at $\omega^{\mathrm{max}}$ is non-degenerate. We do
this by showing that $f'(\tilde \omega) = 0$ for every $\tilde
\omega \in \Omega$. So let $\tilde \omega \in \Omega$ be given. By
Proposition~\ref{p.keane1}, there is a sequence $l_k \to \infty$
such that $T^{l_k} \omega^{\mathrm{max}} \to \tilde \omega$. By
\eqref{eq:condmax}, we can find $\delta_k > 0$ such
 that $|f(\omega) - f(\omega^{\mathrm{max}})| \le \delta_k$ implies
\begin{equation}\label{f.11}
  |\omega - \omega^{\mathrm{max}}| \leq \frac{1}{k}.
\end{equation}
Applying Proposition~\ref{p.keane1} again, we can find a sequence
$n_k \to \infty$ such that $|T^{n_k} \omega^{\mathrm{disc}} -
\omega^{\mathrm{max}}| \leq \frac{\delta_k}{2 \|f'\|_\infty}$ and
hence $|f(T^{n_k} \omega^{\mathrm{disc}}) -
f(\omega^{\mathrm{max}})| \leq \frac{\delta_k}{2} $ for every $k
\geq 1$. Since \eqref{f.maincond} fails, we can find
$\varepsilon_k
> 0$ such that
$|f(T^{n_k} \omega) - f(T^{n_k} \omega^{\mathrm{disc}}) | \le
\frac{\delta_k}{2} $ for every $\omega \in (\omega^{\mathrm{disc}}
- \varepsilon_k, \omega^{\mathrm{disc}})$. Combining these two
estimates, we find $|f(T^{n_k} \omega) - f(\omega^{\mathrm{max}})|
\le \delta_k,  \quad \omega \in (\omega^{\mathrm{disc}} -
\varepsilon_k, \omega^{\mathrm{disc}})$. Thus, by \eqref{f.11}, $
|T^{n_k} \omega - \omega^{\mathrm{max}}| \le \frac{1}{k},  \quad
\omega \in (\omega^{\mathrm{disc}} - \varepsilon_k,
\omega^{\mathrm{disc}}). $ Hence, for any choice of $\omega_k \in
(\omega^{\mathrm{disc}} - \varepsilon_k, \omega^{\mathrm{disc}})$,
we have $T^{n_k} \omega_k \to \omega^{\mathrm{max}}$ and $T^{n_k}
\omega^{\mathrm{disc}} \to \omega^{\mathrm{max}}$ as $k \to
\infty$. Since $T^m$ is continuous at $\omega^{\mathrm{max}}$ for
every $m \geq 1$, we can achieve by passing to a subsequence of
$n_k$ (and hence also of $\omega_k$ and $\varepsilon_k$, but
keeping $l_k$ fixed) that
\begin{equation}\label{f.conv}
T^{n_k + l_k} \omega_k \to \ti{\omega}, \qquad T^{n_k + l_k}
\omega^{\mathrm{disc}} \to \ti{\omega}
\end{equation}
as $k \to \infty$. By possibly making $\varepsilon_k$ smaller, it
follows from Keane's condition that
\begin{equation}\label{f.12}
  \inf_{\omega \in (\omega^{\mathrm{disc}} - \varepsilon_k, \omega^{\mathrm{disc}})}
  |T^{n_k + l_k} \omega - T^{n_k + l_k} \omega^{\mathrm{disc}}| \equiv c(k) >
  0.
\end{equation}
Indeed, for small enough $\varepsilon_k$, $T^{n_k + l_k - 1}:
\ol{T(\omega^{\mathrm{disc}} - \varepsilon_k,
\omega^{\mathrm{disc}})} \to [0,1)$ is a continuous map and
$T\omega^{\mathrm{disc}} \notin \ol{T(\omega^{\mathrm{disc}} -
\varepsilon_k, \omega^{\mathrm{disc}})}$. Hence their images under
$T^{n_k+l_k -1}$ are also disjoint. Now, since \eqref{f.maincond}
fails and \eqref{f.12} holds, we can choose $\omega_k \in
(\omega^{\mathrm{disc}} - \varepsilon_k, \omega^{\mathrm{disc}})$
such that $|f(T^{n_k + l_k} \omega_k) - f(T^{n_k + l_k}
\omega^{\mathrm{disc}})| \leq \frac{c(k)}{k}$ and $|T^{n_k + l_k}
\omega_k - T^{n_k + l_k} \omega^{\mathrm{disc}}| \geq c(k)$ hold.
By \eqref{f.conv}, $T^{n_k + l_k} \omega_k \to \tilde \omega$ and
$T^{n_k + l_k} \omega^{\mathrm{disc}} \to \tilde \omega$. It
therefore follows that
 $$
  |f'(\tilde \omega)| = \lim_{k \to \infty} \frac{|f(T^{n_k + l_k} \omega_k) - f(T^{n_k + l_k} \omega^{\mathrm{disc}})|}
  {|T^{n_k + l_k} \omega_k - T^{n_k + l_k} \omega^{\mathrm{disc}}|}
   \leq \lim_{k \to \infty} \frac{1}{k} = 0.
 $$
This concludes the proof since $\tilde \omega$ was arbitrary.
\end{proof}

\begin{coro}
 For any $T$ not a rotation and satisfying the Keane condition, and
 $\lambda>0$, the Schr\"odinger operator with potential given by $V(n)=\lambda
 \cos (2\pi T^n(\omega))$ has $\sigma_{ac}(H_{\omega})=\emptyset$
\end{coro}
\begin{proof}
 Note that though $\cos(2\pi x)$ does not have a non-degenerate
 maximum (as a function on $[0,1)$), it does have a non-degenerate minimum.
 The argument for this case is the same as in Theorem~\ref{max}.
\end{proof}

\section{Absence of Point Spectrum for Schr\"odinger Operators Generated
by Interval Exchange Transformations}\label{s.pp}

A bounded map $V : \Z \to \R$ is called a \textit{Gordon
potential} if there exists a sequence $q_k \to \infty$ such that
\be
 \limsup_{k\to\infty} \left( \sup_{0 \leq j < q_k} |V(j) - V(j \pm q_k)| \exp(C q_k) \right) = 0,
 \quad \omega\in\Omega_L(T)
\ee for any $C > 0$. It is a well-known fact that Schr\"odinger
operators with a Gordon potential have purely continuous spectrum;
see \cite{g} for the original Gordon result and \cite{d1} for the
result in the form stated above and the history of criteria of
this kind.

We will now identify a dense $G_\delta$ set of interval exchange
transformations such that they generate Gordon potentials for
suitable sampling functions. For the case of two intervals (i.e.,
rotations of the circle), this is a result due to Avron and Simon
\cite{as}. In analogy to the convention for rotations, we will refer
to these as \textit{Liouville} interval exchange transformations.
Explicitly, we call an interval exchange transformation $T$
Liouville if it satisfies the following property: There exists a
sequence $q_k \to \infty$ and $\Omega_L(T)\subseteq\mathbb{T}$ of
full measure such that \be
 \limsup_{k\to\infty} \left( \sup_{0 \leq j < q_k} |T^j \omega - T^{j \pm q_k} \omega| \exp(C q_k) \right) = 0,
 \quad \omega\in\Omega_L(T)
\ee
for any $C > 0$. We will show in the following that the set of Liouville
interval exchange transformation is a dense $G_\delta$ set of zero measure.

If $f: \mathbb{T} \to \R$ is H\"older continuous, the above
property implies that $V_\omega(n) = f(T^n\omega)$ is a Gordon
potential. It was shown by Chulaevskii in \cite{chu} that the set
of Liouville interval exchange transformations is dense and a
$G_\delta$ set. In fact, we state without proof the following
result, which clearly implies that Liouville interval exchange
transformations are dense and $G_\delta$:

\begin{lemma}\label{l.liou}
Given a decreasing function $f: \N \to (0,\infty)$, we can find a
dense $G_\delta$ set $U$ in $\Delta^r$ such that for every
interval exchange transformation in $U$ the following holds.

For almost every $\omega$, there exists a sequence $q_k =
q_k(\omega) \to \infty$ such that for $k \geq 1$,
 \be\label{eq:approp}
  \sup_{0 \leq j < q_k} | T^{j} \omega -  T^{j \pm q_k} \omega | \leq
  f(q_k).
 \ee
\end{lemma}

For the reader desiring further references, one could choose the
interval exchange transformations to be in a small neighborhood of
the primitive interval exchanges considered in \cite{metric} part
II. These were introduced in \cite{kr}. Thus, we have the
following theorem:

\begin{theorem}
If $T$ is a Liouville interval exchange transformation and $f$ is
H\"older continuous, then for every $\omega\in\Omega_L(T)$,
$H_\omega$ has no eigenvalues.
\end{theorem}

We remark that it is easy to show that the set of Liouville
interval exchange transformations has zero measure.

\section*{Acknowledgements}

We thank M.~Boshernitzan, M.~Embree, and S.~Ferenczi for useful
discussions. Corollary~\ref{cor:veech} was pointed out to us
by W.~Veech.
Furthermore, D.~D.\ and H.~K.\ wish to thank the
organizers of the workshop \textit{Low Complexity Dynamics} at the
Banff International Research Station in May~2008, where the work on
this paper was begun.

\end{document}